\documentclass{amsart}

\usepackage{amsmath}
\usepackage{amssymb}
\usepackage{amsthm}
\usepackage{color}
\usepackage{graphicx}
\usepackage{psfrag}

\newcommand{\geqs}{\geqslant}
\newcommand{\leqs}{\leqslant}
\newcommand{\Nb}{{\mathbb N}}
\newcommand{\Rb}{{\mathbb R}}
\newcommand{\Zb}{{\mathbb Z}}

\newcommand{\sgn}{\operatorname{sgn}} 

\renewcommand{\a}{\alpha}
\renewcommand{\b}{\beta}
\renewcommand{\k}{\kappa}
\renewcommand{\t}{\tau}

\providecommand{\abs}[1]{\lvert#1\rvert} 
\providecommand{\norm}[1]{\lVert#1\rVert} 

\newenvironment{Proof}{\removelastskip \vskip12pt plus 1pt \noindent
{\em Proof.\/}\rm }{\hfill$\square$ \vskip12pt plus 1pt}

\newtheorem{theorem}{Theorem}[section]
\newtheorem{lemma}[theorem]{Lemma}
\newtheorem{prop}[theorem]{Proposition}
\newtheorem{definition}[theorem]{Definition}
\newtheorem{coro}[theorem]{Corollary}
\newtheorem{remark}[theorem]{Remark}

\theoremstyle{definition}
\theoremstyle{remark}

\numberwithin{equation}{section}

\begin{document}

\title{the Redner--Ben-Avraham--Kahng cluster system}

\author{F.P. da Costa}
\address{Departmento de Ci\^encias e Tecnologia, Universidade Aberta, Lisboa, Portugal, and
Centro de An\'alise Matem\'atica, Geometria e Sistemas Din\^amicos, Instituto Superior T\'ecnico,
Universidade T\'ecnica de Lisboa, Lisboa, Portugal}
\email{fcosta@uab.pt}
\thanks{Dedicated to Carlos Rocha and Lu\'is Magalh\~aes on the occasion of their sixtieth birthdays.}


\author{J.T. Pinto}
\address{Departmento de Matem\'atica and 
Centro de An\'alise Matem\'atica geometria e Sistemas Din\^amicos, 
Instituto Superior T\'ecnico, Universidade T\'ecnica de Lisboa, Lisboa, Portugal}
\email{jpinto@math.ist.utl.pt}

\author{R. Sasportes}
\address{Departmento de Ci\^encias e Tecnologia, Universidade Aberta, Lisboa, Portugal, and
Centro de An\'alise Matem\'atica, Geometria e Sistemas Din\^amicos, Instituto Superior T\'ecnico,
Universidade T\'ecnica de Lisboa, Lisboa, Portugal}
\email{rafael@uab.pt}


\date{February 22, 2013}



\begin{abstract}
We consider a coagulation model first introduced by Redner, Ben-Avraham and 
Krapivsky in \cite{RBK}, the main feature of which is that the reaction between a $j$-cluster 
and a $k$-cluster results in the creation of a 
$|j-k|$-cluster, and
not, as in Smoluchowski's model, of a $(j+k)$-cluster. 
In this paper we prove existence and uniqueness of solutions under 
reasonably general conditions on the coagulation coefficients, and we also establish 
differenciability properties and continuous dependence of solutions. Some interesting 
invariance properties are also proved. Finally, we study the long-time behaviour of 
solutions, and also present a preliminary analysis of their scaling behaviour.

\end{abstract}

\maketitle


\section{Introduction}

Among the diverse mathematical approaches to modelling the kinetics of cluster growth, one that has received a good deal
of attention consists in the mean field models of coagulation-fragmentation type \cite{LM2004} of which Smoluchowski's
coagulation system is a prototypical case. The basic dynamic process modelled by Smoluchowski's coagulation is the 
binary reaction between
a $j$-cluster (a cluster made up of $j$ identical particles) and a $k$-cluster to produce a $(j+k)$-cluster. So, 
the mean cluster size 
in these coagulating systems tend to increase with time. A contrasting case to coagulation is fragmentation 
in which the basic dynamic
process is the disintegration of a given $j$-cluster into two or more clusters of smaller size. In these 
fragmentation systems mean cluster size 
decreases with time.

A coagulation system that, in spite of its basic mechanism being binary cluster reactions, has 
cluster size evolution similar to that of a fragmentation system is the cluster eating equation. 
This model was introduced by Redner, Ben-Avraham and Kahng in \cite{RBK} (see also \cite{war}) but has received scant attention since. 
We shall call it the Redner--Ben-Avraham--Kahng system (RBK for short).

The basic process is the following: when a $j$-cluster reacts with a $k$-cluster, the result is the production of a
$|j-k|$-cluster (see Fig.~\ref{cluster_eating}).

%
%
\begin{figure}[h]
\begin{center}
\psfrag{j}{$j$-cluster}
\psfrag{k}{$k$-cluster}
\psfrag{j-k}{$|j-k|$-cluster}
\psfrag{+}{$+$}
\includegraphics[scale=.5]{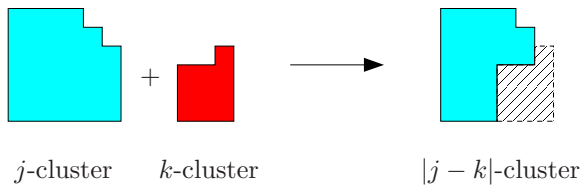}
\end{center}
\caption{Schematic reaction in the cluster eating RBK model}\label{cluster_eating}
\end{figure}
%
%

This process is reminiscent of the coagulation-annihilation model with partial annihilation in
which two or more species of clusters, $A$ and $B$ say, are present, and if a cluster $A_j$ reacts with a cluster $B_k$
the resulting cluster has size $|j-k|$ and is an $A$ cluster if $j>k$, is a $B$ cluster if $j<k$, and is an inert cluster,
neither $A$ or $B$, if $j=k$. Although many problems remain open concerning
coagulation systems with Smoluchowski's type dynamics and with either
complete or incomplete annihilation, there is already a relatively large literature about those systems 
(see, for instance, \cite{CPRS,K,ZY} and references therein). Having just one type of clusters, 
the RBK model could result in a  somewhat easier system to handle mathematically and it is somewhat surprising that, 
to the best of our
knowledge, it has not attracted further attention.

Having in mind the process illustrated in Figure~\ref{cluster_eating}, the index $j$ can no longer represent the total amount of
particles in a cluster (which is a quantity that should be conserved in each elementary reaction) but represents only
the number of  particles in a cluster that are, in some sense, {\it active\/} (a concept whose physical meaning we must leave
undefined.) So, in this paper, every time we refer to a $j$-cluster, we mean a cluster made up with a number $j$ of active 
particles.

Assuming the mass action law of chemical kinetics, the rate of change $\frac{dc_j}{dt}$ of the concentration 
of $j$-clusters at time $t$, $c_j(t)$, has contributions of two 
  different types. It decreases due to reactions of the type $(j)+(k)\rightarrow (j+k)$, 
  for $k\in\Nb$, which corresponds to
$$a_{j,k} c_j(t)c_k(t),$$
where $a_{j,k}$ is the rate coefficient for these equations.
And $\frac{dc_j}{dt}$ increases due to reactions like $(j+k)+(k)\rightarrow (j)$, again with $k\in\Nb,$ 
which have contributions of the type
$$a_{j+k,k}c_{j+k}(t)c_k(t).$$

Adding all these contributions we obtain the Redner--Ben-Avraham--Kahng coagulation system
\begin{equation}
 \frac{dc_j}{dt} = \sum_{k=1}^{\infty}a_{j+k,k}c_{j+k}c_k - \sum_{k=1}^{\infty}a_{j,k}c_jc_k, \qquad j\in\Nb. \label{rbk}
\end{equation}
To simplify notation, we shall often write $W_{p,q}$ instead of $a_{p,q}c_pc_q.$ Naturally, we shall always assume that 
the rate coefficients are symmetric and nonnegative:
\[
 a_{j,k} = a_{k,j},\quad a_{j,k}\geqs 0,\;\forall j, k\in\Nb.
\]

In this paper we study the existence and uniqueness of solutions (\ref{rbk}) in appropriate sequence spaces, 
we investigate some invariance properties of solutions, and also start the study of
their long-time  and scaling behaviours.


\section{Preliminaries}

The mathematical study of (\ref{rbk}) requires the consideration of appropriate spaces. As is 
usual in works in this area, we will consider the Banach spaces
\[
 X_{\mu} := \left\{x=(x_j)\in\Rb^{\Nb}\,|\; \|x\|_{\mu}:=\sum_{j=1}^{\infty}j^{\mu}|x_j| <\infty\right\},\quad \mu\geqs 0,
\]
and their nonnegative cones $X_{\mu}^+:=\{x\in X_{\mu}\,|\; x_j\geqs 0\}.$ Observe that the norm $\|c\|_0$ 
of a given cluster distribution $c$ measures the total amount of clusters that are present. In the usual
coagulation, or coagulation-fragmentation, equations with Smoluchowski's coagulation, the norm $\|c\|_1$
measures the total density or mass of the cluster distribution $c$. Now, with $j$  measuring only the number of 
active particles in a cluster, this norm measures something like an active density or mass. We shall
omit the word ``active'' in what follows.

Being (\ref{rbk}) an infinite dimensional system, we need some care in defining what we mean by a solution.
In this paper we use a definition of solution analogous to the one in \cite{BC} for the standard 
coagulation-fragmentation:

\begin{definition}\label{defsol}
 Let $T\in (0, +\infty].$ A (mild) solution of the Cauchy problem for
 (\ref{rbk}) on $[0, T)$ with initial condition $c(0) = c_0\in X_1^+$ is a function 
 $c=(c_j): [0, T)\to X_1^+$ such that
 \begin{description}
  \item[{\rm (i)}] each $c_j: [0, T)\to \Rb^+$ is continuous and $\displaystyle{\sup_{t\in [0, T)}\|c(t)\|_1<\infty}.$
  \item[{\rm (ii)}] for all $j\in \Nb,$ and all $t\in [0, T)$, 
  we have $\displaystyle{\int_0^t\sum_{k=1}^{\infty}a_{j,k}c_k(s)ds}<\infty.$
  \item[{\rm (iii)}] for all $j\in \Nb,$ the following holds for each $t\in [0, T)$,
  \[
   c_j(t)=c_j(0) + \int_0^t\biggl[\sum_{k=1}^{\infty}a_{j+k,k}c_{j+k}(s)c_k(s) - \sum_{k=1}^{\infty}a_{j,k}c_j(s)c_k(s)\biggr]ds.
  \]
 \end{description}
\end{definition}

\begin{remark}\label{rem1}
 Assuming that, for some nonnegative
 constant $K$ and all positive integers $j$ and $k$, the rate coefficients satisfy
 $a_{j,k}\leqs Kjk,$ then {\rm (i)}$\Rightarrow${\rm (ii)}.\\
 The definition of solution implies that, if $c$ is a solution on $[0, T)$, then each $c_j$ is
 absolutely continuous, so that equation (\ref{rbk}) is satisfied by $c$ a.e. $t\in [0, T)$.
\end{remark}

Also as in works on the standard coagulation-fragmentation equations, we find it convenient to 
consider finite dimensional systems that approximate the infinite dimensional equation (\ref{rbk}). 
This will be particularly relevant for the existence result.

We will now define the finite dimensional approximation of (\ref{rbk}) to be considered in the paper. 
In fact, and in contradistinction 
with the case of Smoluchovski's equation,
 we will prove that, for initial data with compact 
 support, (\ref{rbk}) reduces to this particular finite dimensional system exactly, and so,  
 for that type of initial data, the finite dimensional truncation  is not an approximation at all
but the exact system.  
In fact, it is exactly this compactly supported cases and the corresponding finite dimensional systems that Redner, 
Ben-Avraham and Kahng considered in \cite{RBK}. 

To motivate the finite dimensional system, consider an initial condition for which $c_j(0)=0$ if $j>N$ for some positive $N$. 
Since
the only process is a coagulating one in which the resulting cluster has a smaller size, no clusters with size bigger than $N$ 
can be created.
Mathematically, this is translated in the finite $N$-dimensional system, for an arbitrarily fixed positive integer $N$:   
\begin{equation}
 \frac{dc_j}{dt} = \sum_{k=1}^{N-j}W_{j+k,k}(c) - \sum_{k=1}^{N}W_{j,k}(c), \qquad j\in\{1, \ldots, N\}, \label{frbk}
\end{equation}
where the first sum is defined to be identically zero when $N=1$ or if $j=N.$

Naturally, since (\ref{frbk}) is a finite dimensional system with a polynomial right hand side (as a function of the components 
$c_j$ of the solution vector
$c=\bigl(c_j(\cdot)\bigr)$\/), the existence of local solutions 
to the Cauchy problems follows immediately from the standard Picard-Lindel\"of existence
theorem. In the following proposition we collect 
basic results
about solutions to this finite dimensional system.

\begin{prop}\label{prop2-1}
Let $c=\bigl(c_j(\cdot)\bigr): I_{\rm max}\to\Rb^{N^+}$ be the unique local solution of (\ref{frbk}) 
with initial condition $c(0)=c_{0}$ and let $I_{\rm max}$ be its
maximal interval. Then
\begin{description}
 \item[{\rm (i)}] For every sequence $(g_j),$ and every $m \in\{1, \ldots, N\}$ the following holds
\begin{equation}
 \sum_{j=m}^Ng_j\dot{c}_j = -\sum_{T_1}(g_j-g_{j-k})W_{j,k} - \sum_{T_2}g_jW_{j,k}, \label{wfrbk}
\end{equation}
where 
\[
\begin{array}{l}
 T_1=T_1(m,N):=\{(j,k)\in\{m,\ldots, N\}\times\{1,\ldots, N\}|\; k\leqs j-m\},  \\
 T_2=T_2(m,N):=\{(j,k)\in\{m,\ldots, N\}\times\{1,\ldots, N\}|\; k\geqs j-m+1\}. 
\end{array}
\]   
 \item[{\rm (ii)}] If all components of the initial condition $c_0$ are nonnegative, then also $c_j (t)\geqs 0,$ 
for all $j\in\{1,\ldots, N\}$ and all $t\in I\cap\Rb^+,$  where $\Rb^+:=[0, +\infty).$ 
\item[{\rm (iii)}] $\sup I_{\rm max} = +\infty.$ 
\end{description}
\end{prop}

\begin{proof} The proofs of these results follow the corresponding ones for the usual coagulation-fragmentation equation closely
(see \cite[Lemmas 2.1, 2.2]{BC}.)
\begin{description}
\item[{\rm (i)}] Multiplying equation (\ref{frbk}) by $g_j$ and summing in $j$ from $m$ to $N$ one gets
\[
 \sum_{j=m}^Ng_j\dot{c}_j = \sum_{j=m}^{N-1}\sum_{k=1}^{N-j}g_jW_{j+k,k} - \sum_{j=m}^{N}\sum_{k=1}^{N}g_jW_{j,k},
\]
and now an easy manipulation (a change of notation in the first sum and the separation of the second sum into a sum over $T_1$ 
and another over $T_2$) gives the result (\ref{wfrbk}).
\item[{\rm (ii)}] Write (\ref{frbk}) as $\dot{c}_j = R_j(c)-c_j\varphi_j(c)$ where
\[
 R_j(c):=\sum_{k=1}^{N-j}W_{j+k,k}(c),\qquad \varphi_j(c):=\sum_{k=1}^Na_{j,k}c_k.
\]
Suppose that, for some $\tau\in I\cap\Rb^+$ and all $j=1, \ldots, N$, we have $c_j(\tau)\geqs 0$ and $c_r(\tau)=0$ for some $r$.
For $\varepsilon >0$ consider the initial value problem
\[
 \begin{array}{l}
  \dot{c}_j^{\varepsilon} = R_j(c^{\varepsilon}) - c_j^{\varepsilon}\varphi_j(c^{\varepsilon})+ \varepsilon\\
  c_j^{\varepsilon}(\tau)\rule{0mm}{4mm} = c_j(\tau).
 \end{array}
\]
Thus, $\dot{c}_r^{\varepsilon}(\tau) = R_r(c^{\varepsilon}(\tau)) + \varepsilon > 0$ and, for some $\eta>0$, $c_r^{\varepsilon}(t)>0$, 
for
$t\in (\tau, \tau+\eta),$ for each $\varepsilon\in (0, \varepsilon_0)$, with $\varepsilon_0$ fixed. By continuous dependence, by
making
$\varepsilon \downarrow 0,$ we conclude that $c_r^{\varepsilon}(t)\to c_r(t), t\in [\tau, \tau+\eta),$ and thus 
$c_r(t)\geqs 0,  t\in [\tau, \tau+\eta).$
Hence, we conclude the nonnegativity of each $c_j$ for $t\in I\cap\Rb^+.$
\item[{\rm (iii)}] Using the expression (\ref{wfrbk}) proved in (i) and the nonnegativity of solutions in (ii) we conclude that, 
for every nondecreasing positive sequence $(g_j)$,
we have, for all $m\in\{1, \ldots, N\},$
\[
 \sum_{j=m}^Ng_j\dot{c}_j(t) \leqs 0.
\]
But then $c$ is bounded and, being the right-hand side of (\ref{frbk}) bounded in bounded subsets of $\Rb^{N},$ we conclude 
that $\sup I_{\rm max} = +\infty.$
\end{description}
This concludes the proof.
\end{proof}

It is easy to conclude from the equations (\ref{rbk}), (\ref{frbk}) and the Definition~\ref{defsol}, 
that if $c^N(\cdot)$ is a solution of (\ref{frbk}), then the function 
$c:=(c_1^N, c_2^N, \ldots, c_N^N, 0, 0, \ldots)$ is a solution of (\ref{rbk}).

Also easy to conclude, by choosing $g_j\equiv j$ in Proposition~\ref{prop2-1}-{\rm (i)}, is the fact 
that the density of solutions
to the finite dimensional systems (\ref{frbk}) decreases with time, i.e., for all $t_2\geqs t_1,$ we have
$\|c(t_2)\|_1\leqs \|c(t_1)\|_1.$ Given the type of coagulation
process under consideration (remember Figure~\ref{cluster_eating}), this is a physically reasonable behaviour for 
solutions to the infinite dimensional system (\ref{rbk}). However, it is not presently clear that other types of 
(nonphysical) solution can not, in fact, exist, similarly to what happens in the case of the pure fragmentation equation 
\cite[Example 6.2]{BC}. 

We end this section by introducing the following definition:

\begin{definition}\label{admsolnds}
 Let $c$ be a solution to (\ref{rbk}).
 \begin{description}
  \item[{\rm (i)}] we call $c$ an {\em admissible solution\/} if it can be obtained as the uniform
  limit in compact sets of $[0, \infty)$, as $N\to\infty$, 
  of a sequence of solutions $c^N$ to  (\ref{rbk}) such that $c_j^N\equiv 0, \forall j> N.$ 
  (In particular, $(c^N_1, \ldots, c_N^N)$ can be a solution to (\ref{frbk}).)
  \item[{\rm (ii)}] we call $c$ a {\em density nonincreasing solution\/} if, for all $t_2\geqs t_1,$ it holds 
$\|c(t_2)\|_1\leqs \|c(t_1)\|_1.$
 \end{description}
\end{definition}

\begin{remark}
In the literature of coagulation-fragmentatiom equations an admissible solution is one that can be obtained as the weak limit as $n\to\infty$ 
of a sequence to finite $n$-dimensional truncations of the system \cite{CdC}. In Definition~\ref{admsolnds} we impose the condition of
uniform convergence in compact subsets of time $t$. This corresponds to what we can prove in this case (see Corollary~\ref{admis} below); it is also what
happens to be the case for coagulation-fragmentation equations with coagulation kernels growing at most linearly \cite[Corollary 2.6]{BC}.
\end{remark}


\section{Existence of solutions}

In this section we shall prove existence of solutions in $X^+_1$ of Cauchy problems for (\ref{rbk}) 
with initial data in $X^+_1,$ with some mild conditions on the coagulation coefficients. 

\begin{theorem}\label{theoexist}
 Assume $a_{j,k}\leqs Kjk,$ for some positive constant $K$ and all positive integers $j$ and $k$.
 Let $c_0\in X_1^+.$ Then, there is at least one solution of (\ref{rbk}) with initial condition 
 $c(0)=c_0$, defined on $[0, T)$, for some $T\in (0, +\infty].$
\end{theorem}

\begin{proof}
As is usual in coagulation studies \cite{BC,L2002} the proof is based on passing to the limit \(N\to\infty\) in a 
sequence of solutions to the $N$-dimensional system (\ref{frbk}), which we do by an application of Helly's selection theorem,
and then by proving that the limit sequence is a solution to the infinite dimensional system (\ref{rbk}). In order to do 
this we need some bounds on the moments of the solutions to the finite dimensional systems. Actually, for the RBK system, the 
application of the method just described is much easier than in \cite{BC,L2002}  for the coagulation-fragmentation system with Smoluchowski
coagulation due to the {\it a priori\/} estimate (\ref{normp}) on the zeroth and first moments. 

 Let $c^N$ be the solution to the finite $N$-dimensional system (\ref{frbk}) satisfying the initial condition
 $c^N_j(0)=c_{0j},$ for $j\in\{1, \ldots, N\}$. By putting $g_j\equiv 1$ and $g_j\equiv j$ in (\ref{wfrbk})
 we immediately conclude that, for $1\leqs m\leqs N,$
 \[
  \sum_{j=m}^N\dot{c}^N_j \leqs 0,\quad \text{and}\quad\sum_{j=m}^Nj\dot{c}^N_j \leqs 0, 
 \]
respectively. Thus, for $p=0, 1$ we have
\begin{equation}
 \displaystyle{\sum_{j=m}^N}j^pc_j^N  \leqs  \displaystyle{\sum_{j=m}^N}j^pc_{0j}^N  =  \displaystyle{\sum_{j=m}^N}j^pc_{0j} \leqs 
 \displaystyle{\sum_{j=m}^\infty}j^pc_{0j}\leqs \displaystyle{\sum_{j=1}^\infty}j^pc_{0j}  =  \|c_0\|_p. \label{normp} 
\end{equation}
Let us now prove that $(c^N)$ is uniformly bounded in $W^{1,1}(0, T)$, for all fixed $T\in (0, \infty).$
From the definition of the norms in \(X_0\) and \(X_1\), and from (\ref{normp}) with $p=0$ we immediately get 
\begin{equation}
 \|c^N_j\|_{L^1(0,T)} = \int_0^T\!\!|c^N_j(s)|ds \leqs \int_0^T\!\!\|c^N(s)\|_0ds \leqs  \int_0^T\!\!\|c_0\|_0ds = T\|c_0\|_0 \leqs T\|c_0\|_1 \label{boundC0}
\end{equation}
By equation (\ref{frbk}) we have
\begin{eqnarray}
\left\|\frac{dc_j^N}{dt}\right\|_{L^1(0,T)} &=& \int_0^T\left|\frac{dc_j^N}{dt}(s)\right|ds  \nonumber \\
& \leqs & 
\int_0^T\sum_{k=1}^{N-j}W_{j+k,k}(c^N(s))ds + \int_0^T \sum_{k=1}^{N}W_{j,k}(c^N(s))ds \label{boundC1-1}
\end{eqnarray}
Estimating the first integral in (\ref{boundC1-1}) we obtain
\begin{eqnarray*}
\int_0^T\sum_{k=1}^{N-j}W_{j+k,k}(c^N(s))ds & \leqs & K\int_0^T\sum_{k=1}^{N-j}jkc_{j+k}^N(s)c_k^N(s)ds\\
& \leqs & K\int_0^T\sum_{k=1}^{N-j}(j+k)c_{j+k}^N(s)\sum_{k=1}^{N-j}kc_k^N(s)ds\\
& \leqs & K\int_0^T\|c^N(s)\|_1^2ds\; \leqs \; K\int_0^T\|c_0\|_1^2ds \\
& = & KT\|c_0\|^2_1,
\end{eqnarray*}
and for the second integral in (\ref{boundC1-1}) we get
\begin{eqnarray*}
\int_0^T\sum_{k=1}^{N}W_{j,k}(c^N(s))ds & \leqs & K\int_0^T\sum_{k=1}^{N}jkc_{j}^N(s)c_k^N(s)ds\\
& \leqs & K\int_0^T\sum_{k=1}^{N}jc_{j}^N(s)\sum_{k=1}^{N}kc_k^N(s)ds\\
& \leqs & K\int_0^T\|c^N(s)\|_1^2ds\; \leqs \; K\int_0^T\|c_0\|_1^2ds \\
& = & KT\|c_0\|^2_1.
\end{eqnarray*}
Thus, substituting in (\ref{boundC1-1}) we conclude that
\begin{equation}
 \left\|\frac{dc_j^N}{dt}\right\|_{L^1(0,T)} \leqs 2KT\|c_0\|^2_1
\end{equation}
and therefore
\[
 \|c^N_j\|_{L^1(0,T)} + \left\|\frac{dc_j^N}{dt}\right\|_{L^1(0,T)} \leqs (1+2K\|c_0\|_1)T\|c_0\|_1.
\]

So, by Helly's selection theorem, for each fixed $j$
there exists a subsequence of $(c^N_j)_N$ (not relabeled), converging pointwise to a BV function in $[0, T]$, $c_j(\cdot),$
\[
 c_j^N(t)\rightarrow c_j(t),\quad\text{as}\quad N\to\infty,\; \forall t\in [0, T], \;\forall j\in\Nb.
\]
But then, for each $q\in\Nb$, and for each $t\in\left[0,T\right]$,
\[
  \sum_{j=1}^qjc^N_j (t) \;\rightarrow\;   \sum_{j=1}^qjc_j (t),\;\text{ as }N\to\infty\,,
\]
and therefore, by (\ref{normp}), for any such $q$,  $ \sum_{j=1}^qjc_j (t)\leqs \|c_{0}\|_{1}$. By making $q\to\infty$, we obtain
\begin{equation}\label{est1norm}
 \sum_{j=1}^\infty jc_j (t)\leqslant\|c_{0}\|_{1}\,.
\end{equation}
Since proposition \ref{prop2-1}(ii) implies $c_{j}(t)\geqs 0$,  this proves that, not only $c(t)\in X^+_{1}$, for each $t\in\left[0,T\right]$, but also 
that condition (i) of definition \ref{defsol} is fulfilled.

It remains to be proven that the limit functions $c_j$ solve the RBK system (\ref{rbk}). In order to obtain this result,
we shall pass to the limit \(N\to\infty\) in the equation for $c_j^N$,
\[
 c_j^N(t) = c_{0 j} + \int_0^t\sum_{k=1}^{N-j}W_{j+k,k}(c^N(s))ds - \int_0^t \sum_{k=1}^{N}W_{j,k}(c^N(s))ds.
\]
Thus, we need to  prove that, for all $t\in[0, T],$
\begin{equation}
 \int_0^t \sum_{k=1}^{N}W_{j,k}(c^N(s))ds \xrightarrow[N\rightarrow \infty]{} \int_0^t \sum_{k=1}^{\infty}W_{j,k}(c(s))ds, \label{conv1}
\end{equation}
and
\begin{equation}
 \int_0^t\sum_{k=1}^{N-j}W_{j+k,k}(c^N(s))ds \xrightarrow[N\rightarrow \infty]{} \int_0^t\sum_{k=1}^{\infty}W_{j+k,k}(c(s))ds. \label{conv2}
\end{equation}
The proofs of (\ref{conv1}) and (\ref{conv2}) are entirely analogous, and so we shall present only the proof of (\ref{conv1}), leaving
the details of the other to the reader.

We first start by proving that the right-hand side of (\ref{conv1}) is well defined. Let $p$ be an arbitrarily fixed positive integer.
By the definition of $(c_j)$ we know that
\[
 \sum_{k=1}^{p}a_{j,k}c_j^Nc_k^N \xrightarrow[N\rightarrow \infty]{} \sum_{k=1}^{p}a_{j,k}c_jc_k,
\]
and from (\ref{normp}) we have that, for all positive integers $N$ and $p$,
\[
 \sum_{k=1}^{p}a_{j,k}c_j^Nc_k^N \leqs K\|c_0\|_1^2,
\]
and thus also
\[
 \sum_{k=1}^{p}a_{j,k}c_jc_k \leqs K\|c_0\|_1^2.
\]
Consequently, since the right-hand side is independent of $p$ and all the terms are nonnegative,
\[
 \sum_{k=1}^{\infty}a_{j,k}c_jc_k \leqs K\|c_0\|_1^2,
\]
 and the dominated convergence theorem implies that, for all $t\in (0, T)$, with $T<\infty$, the right-hand side of (\ref{conv1}) is well defined.

Now we shall prove the limit in (\ref{conv1}) holds.
Let $m$ be a positive integer such that \(1\leqs m < N <\infty\) but otherwise arbitrarily fixed. Then
\begin{eqnarray}
 \lefteqn{\left|\displaystyle{\int_0^t\sum_{k=1}^Na_{j,k}c_j^N(s)c_k^N(s)ds - \int_0^t\sum_{k=1}^{\infty}a_{j,k}c_j(s)c_k(s)ds}\right| \leqs}\nonumber \\
&\leqs & \quad\displaystyle{\int_0^t\sum_{k=1}^{m-1}a_{j,k}\left|c_j^N(s)c_k^N(s)-c_j(s)c_k(s)\right|ds} +   \label{convfin}\\
& & + \displaystyle{\int_0^t\sum_{k=m}^{N}a_{j,k}c_j^N(s)c_k^N(s)ds} + 
      \displaystyle{\int_0^t\sum_{k=m}^{\infty}a_{j,k}c_j(s)c_k(s)ds},\label{convrest}
\end{eqnarray} 
and we need to prove that the right-hand side of this inequality can be made arbitrarily small when $N\to\infty,$ by choosing 
$m$ sufficiently large.

Since each term in the sum in (\ref{convfin}) converges pointwise to zero, the sum has a finite fixed number of terms, 
and its absolute value is bounded above by $2K\|c_0\|_1^2$, the dominated convergence theorem implies that (\ref{convfin}) converges to zero as $N\to\infty$.

Let us now consider the integrals in (\ref{convrest}). Define
$\rho_m := \|c_0\|_1\sum_{j=m}^{\infty}jc_{0 j}.$ Clearly $\rho_m\to 0$ as $m\to\infty.$

From (\ref{normp}) we conclude that
\begin{eqnarray}
 \displaystyle{\int_0^t\sum_{k=m}^{N}a_{j,k}c_j^N(s)c_k^N(s)ds} & \leqs & K\displaystyle{\int_0^t\sum_{k=m}^{N}jc_j^N(s)kc_k^N(s)ds} \nonumber\\
& \leqs & K\displaystyle{\int_0^t \|c_0\|_1\sum_{k=m}^{N}kc_k^N(s)ds \;\,\leqs\;\, K\int_0^t\rho_mds}  \nonumber\\
& = & KT\rho_m,
\end{eqnarray}
and so we get the first integral in (\ref{convrest}) can be made arbitrarily small by choosing $m$ sufficiently large. 
For the second integral the result is proved in an analogous way:
For all $1\leqs m < p$ we have
\[
 \sum_{k=m}^pa_{j,k}c_j^Nc_k^N \xrightarrow[N\rightarrow \infty]{} \sum_{k=m}^pa_{j,k}c_jc_k.
\]
Due to (\ref{normp}), the sum in the left-hand side is bounded by \(K\rho_m\), and so we also get
\[
 \sum_{k=m}^pa_{j,k}c_jc_k \leqs K\rho_m,
\]
for all $p$. Since this bound is uniform in $p$, we have
\[
 \sum_{k=m}^pa_{j,k}c_jc_k\xrightarrow[p\rightarrow \infty]{} \sum_{k=m}^{\infty}a_{j,k}c_jc_k \leqs K\rho_m.
\]
Hence, by the dominated convergence theorem, the second integral in (\ref{convrest}) can also be made arbitrarily small by choosing
$m$  and $N$ sufficiently large.

This completes the proof of (\ref{conv1}). As pointed out above, the proof of (\ref{conv2}) is entirely analogous and will be omitted.
\end{proof}

\begin{coro}\label{admis}
The solution obtained in Theorem \ref{theoexist} can be extended to $t\in [0,+\infty[$  as an admissible solution.  
\end{coro}
\begin{proof}
The uniform convergence property is again a consequence of (\ref{normp}). In fact, by applying to (\ref{normp}) an argument 
 similar to the one that led us to (\ref{est1norm}) we obtain, for each $m,N\in\Nb$, $t\in[0,T]$,
\[
\sum_{j=m}^{\infty}j|c_{j}(t)-c_{j}^N(t)|\leqs 2\sum_{j=m}^{\infty}jc_{0j}\,.
\]
Since $\sum_{j=m}^{\infty}jc_{0j}\to 0$, as $m\to\infty$, we conclude that the series 
in the l.h.s. of this inequality with $m=1$ is convergent uniformly in 
$(t,N)\in[0,T]\times\Nb$. Since, for each $j$ and $t$, $j|c_{j}(t)-c_{j}^N(t)|\to 0$, 
as $N\to\infty$,  we conclude that, as $N\to\infty$,
\[
 \sum_{j=1}^\infty jc_j^N (t)\;\to\; \sum_{j=1}^\infty jc_j (t)\,,
\]
as $N\to\infty$, uniformly in $t\in[0,T]$.

That $c(\cdot)$ is extendable to $[0,+\infty[$ is a consequence of the arbitrariness of $T>0$ 
and estimate (\ref{est1norm}).
\end{proof}


\section{The moments' equation}
As in the studies of the usual coagulation-fragmentation systems, a weak formulation of (\ref{rbk}) is a tool of the utmost
importance. This weak version, presented next, is the version of the expression (i) of Proposition~\ref{prop2-1}, 
written for (\ref{rbk}) instead of (\ref{frbk}).


\begin{prop}\label{prop2-2}
Let $c=\bigl(c_j(\cdot)\bigr): I_{\rm max}\to\Rb^N$ be a solution of (\ref{rbk}) 
 and let $\tau, t\in I_{\rm max}$ be such that $\tau\leqs t.$ For every sequence $(g_j),$ and all
positive integers $m$ and $n$ with $m<n$ the following {\em moment's equation\/} holds

\begin{multline}
\sum_{j=m}^ng_jc_j(t)- \sum_{j=m}^ng_jc_j(\tau) \;=\\
 = - \int_{\tau}^t\sum_{S_1}(g_j-g_{j-k})W_{j,k} - \int_{\tau}^t\sum_{S_2}g_jW_{j,k}
+ \int_{\tau}^t\sum_{S_3}g_{j-k}W_{j,k}. \label{momeq}
\end{multline}
where 
\[
\begin{array}{l}
 S_1=S_1(m,n):=\{(j,k)\in\Nb^2|\; m+1\leqs j< n+1,\, 1\leqs k\leqs j-m\},  \\
 S_2=S_2(m,n):=\{(j,k)\in\Nb^2|\; m\leqs j< n+1,\, k\geqs j-m+1\}\\
 S_3=S_3(m,n):=\{(j,k)\in\Nb^2|\; j\geqs n+1,\, j-n\leqs k\leqs j-m\}.
\end{array}
\]    
\end{prop}

In Fig.~\ref{regions} we give a geometric representation of the regions $S_j$.
%
%
\begin{figure}[ht]
\begin{center}
\psfrag{j}{$j$}
\psfrag{k}{$k$}
\psfrag{1}{$1$}
\psfrag{n}{$n$}
\psfrag{n+1}{$n+1$}
\psfrag{m}{$m$}
\psfrag{m+1}{$m+1$}
\psfrag{S1}{$S_1$}
\psfrag{S2}{$S_3$}
\psfrag{S3}{$S_2$}
\includegraphics[scale=.35]{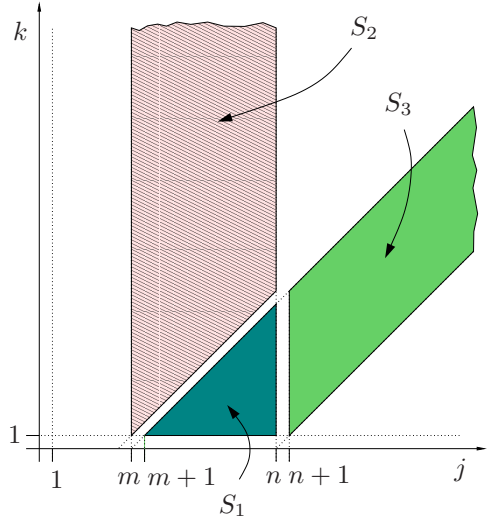}
\end{center}
\caption{Regions $S_j$ defined in Proposition~\ref{prop2-2}}\label{regions}
\end{figure}
%
%


\begin{lemma}\label{momlem}
Suppose the coefficients  \(a_{j,k}\) satisfy the condition \(a_{j,k}\leqslant Kjk \),  
then
\begin{equation}
  \lim_{m\to \infty}\int_{\tau}^{t}m \sum_{S_2(m,\infty)}W_{j,k} ds =0,\label{limT2}
\end{equation}
\end{lemma}
\begin{Proof}
Considering \(g_j\equiv 1 \) in the moments' equation \eqref{momeq}  we obtain
\begin{equation}
\sum_{j=m}^nc_j(t)- \sum_{j=m}^nc_j(\tau)
 = - \int_{\tau}^t\sum_{S_2}W_{j,k}
+ \int_{\tau}^t\sum_{S_3}W_{j,k}.\label{momj1}
\end{equation}
We start by estimating the expression in the \(S_3\) region. We clearly have
 \[
  \sum_{S_3(m,n)}W_{j,k} \leqslant  K \sum_{j=n+1}^{\infty}\sum_{k=j-n}^{j-m}(jc_j)(kc_k)\leqslant
   K  \norm{c(t)}_1^2,
 \]
 and also
\(
 \displaystyle{\sum_{S_3(m,n)}W_{j,k} \leqslant K\norm{c(t)}_1 \sum_{j=n+1}^{\infty}jc_j(t)\to 0}, 
\)
\(\text{pointwise as } n\to \infty\). 
Since, from the definition of solution, \(\norm{c(t)}_1 \) is bounded in \([\tau,t] \), applying the dominated convergence theorem gives
\[
 \int_{\tau}^{t}\sum_{S_3(m,n)}W_{j,k}ds \to 0, \text{ as } n\to \infty .
\]
In \(S_2\) the same bound \( \displaystyle{\sum_{S_2(m,n)}W_{j,k} \leqs K \norm{c(t)}_1^2}\)  is true
and the convergence
\[
\sum_{j=m}^{n}\sum_{k=j-m+1}^{\infty}W_{j,k} \to  \sum_{j=m}^{\infty}\sum_{k=j-m+1}^{\infty}W_{j,k},
\]
is valid pointwise in $t$ as \( n\to \infty \). Hence, again by the dominated convergence theorem,
\[
 \int_{\tau}^{t} \sum_{S_2(m,n)} W_{j,k}  \to  \int_{\tau}^{t} \sum_{S_2(m,\infty)} W_{j,k} , \text{ as }  n\to \infty.
\]
Now taking limits, as \( n\to\infty \) on both sides of \eqref{momj1} we obtain
\begin{equation*}
\sum_{j=m}^{\infty}c_j(t)- \sum_{j=m}^{\infty}c_j(\tau)
 = - \int_{\tau}^t\sum_{S_2(m,\infty)}W_{j,k}.
\end{equation*}

But \(m\sum_{j=m}^{\infty}c_j\leqslant \sum_{j=m}^{\infty}j c_j \to 0 \text{ as } m\to \infty\), since by definition of solution, \(c\in X_1 \), we then obtain 
\begin{equation}
  \lim_{m\to \infty}\int_{\tau}^{t}m \sum_{S_2(m,\infty)}W_{j,k} ds =0.
\end{equation}
This concludes the proof.
\end{Proof}

Another estimate that will be useful is the following


\begin{prop}
Suppose the coefficients  \(a_{j,k}\) satisfy the condition \(a_{j,k}\leqslant Kjk \), and  the sequence  \((g_j)\) satisfies  \(| g_j| \leqslant j \) 
then, for each $m\in\Nb$,
\begin{equation}
 \sum_{j=m}^{\infty}g_jc_j(t) - \sum_{j=m}^{\infty} g_jc_j(\tau)=-\lim_{n\to \infty}\int_{\tau}^{t}\biggl[ \sum_{S_1(m,n)}(g_j -g_{j-k})W_{j,k} +
 \sum_{S_2(m,n)}g_jW_{j,k} \biggr].\label{limSSS}
\end{equation}
Furthermore, with the stronger assumptions \(a_{j,k}\leqslant K(jk)^\beta \) with \(0\leqslant \beta \leqslant \frac{1}{2} \), and the sequence 
\((g_j)\) satisfying \(| g_j| \leqslant j \)  and \(|g_{j}-g_{k}|\leqs M|j-k|\), for all \(j \) and \(k\), and for some positive constant $M$,
the following holds true:
\begin{equation}
 \sum_{j=m}^{\infty}g_jc_j(t) - \sum_{j=m}^{\infty} g_jc_j(\tau)=-\int_{\tau}^{t}
 \sum_{S_{1}(m,\infty)}(g_j - g_{j-k}) W_{j,k} -\int_{\tau}^{t}\sum_{S_{2}(m,\infty)} g_jW_{j,k}.\label{weakmoment}
\end{equation}
\end{prop}
\begin{proof}
In the moments' equation \eqref{momeq} we prove that 
\begin{equation}\label{limS3}
\int_{\tau}^{t}\sum_{S_3(m,n)}g_{j-k}W_{j,k}\to 0\quad \text{as}\quad n\to\infty\,.
\end{equation}
In fact, we observe that \( S_3(m,n)\subset S_2(n+1,\infty)\), and since in  \( S_3(m,n) \) it holds that \( |j-k|=j-k \leqslant n\), we get,
by the previous lemma,  
\begin{multline*}
 0\leqslant \biggl|\int_{\tau}^{t}\sum_{S_3(m,n)}g_{j-k}W_{j,k}\biggr|\leqslant
 \int_{\tau}^{t}\sum_{S_3(m,n)}|j-k|W_{j,k}\leqslant\\
 \leqslant (n+1)\int_{\tau}^{t}\sum_{S_2(n+1,\infty)}W_{j,k} \to 0, \text{ as } n\to \infty,
\end{multline*}
thus proving \eqref{limS3}.
As a consequence, by taking the limit as \(n\to \infty\) in 
 \eqref{momeq}, we obtain \eqref{limSSS}.
Now, by imposing the stronger conditions of the second part of the proposition, we have, for each \(n\in\Nb\),
\begin{equation}\label{S2gj}
\biggl|\sum_{S_{2}(m,n)} g_jW_{j,k}\biggr|\leqs K\sum_{S_{2}(m,n)}j(jk)^\beta c_{j}c_{k}
\leqs  K\sum_{S_{2}(m,n)}jk c_{j}c_{k}
\leqs K\|c\|_{1}\sum_{j=m}^n jc_{j},
\end{equation}
where in the second inequality we have used the fact that, if \((j,k)\in S_{2}(m,n)\), then \(j\leqs k\)\,, and so
\((jk)^{\beta}\leqs k^{2\beta}\leqs k,\) due to the assumption $\beta\leqs \frac{1}{2}.$ Similarly,  
for $(j,k)$ in \(S_{1}(m,n)\) we have
\(k< j\), and thus
\begin{multline}\label{S1gj}
\sum_{S_1(m,n)}|g_j -g_{j-k}|W_{j,k}\leqs  MK\sum_{S_1(m,n)}k(jk)^\beta c_{j}c_{k}\\
\leqs MK\sum_{S_1(m,n)}jk c_{j}c_{k}
\leqs MK\|c\|_{1}\sum_{j=m}^n jc_{j}.
\end{multline}
Estimates (\ref{S2gj}) and (\ref{S1gj}) and (i) 
from the Definition \ref{defsol}, together with the dominated convergence theorem allow us to prove (\ref{weakmoment}).
\end{proof}


\begin{coro}
Suppose the coefficients  \(a_{j,k}\) satisfy the condition \(a_{j,k}\leqslant Kjk \),
then any solution is a density nonincreasing solution. Moreover, with the stronger assumption that 
\(a_{j,k}\leqslant K(jk)^\beta \) with \(0\leqslant \beta \leqslant \frac{1}{2} \), the following holds true
\begin{align}
\sum_{j=1}^{\infty}jc_j(t) - \sum_{j=1}^{\infty} jc_j(\tau)&=
-2\int_{\tau}^{t}\sum_{j=1}^{\infty}\sum_{k=j+1}^{\infty}j W_{j,k}  -\int_{\tau}^{t}\sum_{j=1}^{\infty} jW_{j,j},\label{jwjj}\\
\sum_{j=1}^{\infty}c_j(t) - \sum_{j=1}^{\infty} c_j(\tau)&=
-\frac{1}{2}\int_{\tau}^{t}\sum_{j=1}^{\infty}\sum_{k=1}^{\infty} W_{j,k}  -\frac{1}{2}\int_{\tau}^{t}\sum_{j=1}^{\infty} W_{j,j}.\label{varc}
\end{align}
\end{coro}
Equations \eqref{jwjj} and \eqref{varc} follow from \eqref{weakmoment} considering \(g_j =j \) and \(g_j =1 \).

An interesting  particular case concerns the evolution of the number of clusters of odd size, that we will consider in Section~9.


\begin{coro}
Choosing 
\(
g_{j}=\delta_{j,\text{\emph{odd}}}
\)
then
\begin{equation}
\sum_{\substack{ j=1\\j \text{ \emph{odd}}}}^{\infty}c_j(t) - 
\sum_{\substack{ j=1\\j \text{ \emph{odd}}}} ^{\infty}c_j(\tau)=-\int_{\tau}^{t}\sum_{\substack{ j=1\\j \text{ \emph{odd}}}}^{\infty}
\sum_{\substack{ k=1\\k \text{ \emph{odd}}}}^{\infty}
W_{j,k}.\label{oddweak}
\end{equation}
\end{coro}
\begin{Proof}
Considering this choice of the sequence \( (g_{j})\) in \eqref{weakmoment} we have, after some rearrangements
\begin{equation*}
\sum_{\substack{ j=1\\j \text{ odd}}}^{\infty}c_j(t) - \sum_{\substack{ j=1\\j \text{ odd}}}^{\infty} c_j(\tau)=
-\int_{\tau}^{t}\sum_{\substack{S_{1}\cup S_{2}\\j \text{ odd}\\k \text{ odd}} } W_{j,k}-\int_{\tau}^{t}\sum_{\substack{ S_{2}\\j \text{ odd}\\k \text{ even}} } W_{j,k}
+\int_{\tau}^{t}\sum_{\substack{S_{1}\\j \text{ even}\\k \text{ odd}}} W_{j,k}.
\end{equation*}
The last two terms cancel out, since using the fact that \( W_{j,k} = W_{k,j}\) we have
\begin{align*}
\sum_{\substack{ S_{2}\\j \text{ odd}\\k \text{ even}} } W_{j,k}&=\sum_{\substack{ j=2\\j \text{ odd}}}^{\infty} \sum_{\substack{ k=1\\k \text{ even}}}^{\infty} W_{j,k}
=\sum_{\substack{ k=1\\k \text{ even}}}^{\infty} \sum_{\substack{ j=k+1\\j \text{ odd}}}^{\infty} W_{j,k}
=\sum_{\substack{ j=1\\j \text{ even}}}^{\infty} \sum_{\substack{ k=j+1\\k \text{ odd}}}^{\infty} W_{j,k},\\
\intertext{and similarly in the \(S_{1}\) region}
\sum_{\substack{S_{1}\\j \text{ even}\\k \text{ odd}}} W_{j,k}&=\sum_{\substack{ j=1\\j \text{ even}}}^{\infty} \sum_{\substack{ k=j\\k \text{ odd}}}^{\infty} W_{j,k}=\sum_{\substack{ j=1\\j \text{ even}}}^{\infty} \sum_{\substack{ k=j+1\\k \text{ odd}}}^{\infty} W_{j,k}, \text{ since \( j \) is even and \( k \) is odd}.
\end{align*}
Since \(\displaystyle{
\sum_{\substack{S_{1}\cup S_{2}\\j \text{ odd}\\k \text{ odd}} } W_{j,k}=\sum_{\substack{ j=1\\j \text{ odd}}}^{\infty} \sum_{\substack{ k=1\\k \text{ odd}}}^{\infty} W_{j,k}}\), 
this concludes the proof.
\end{Proof}


\section{A uniqueness result}

We now consider a uniqueness result for (\ref{rbk}). The result is obtained by assuming the initial value problem 
has two solutions and proving they are equal. This will be done, as usual in coagulation problems (see, e.g.,
\cite{BC,BCP}) by appropriate estimates on the solutions and the use of Gronwall's inequality. 
The proof requires conditions on the coagulation coefficients that are
slightly more restrictive than the ones used for the existence result. At the time of writing it is not clear if these
conditions can be significantly relaxed.

\begin{prop}\label{uniqueness}
 Let \(a_{j,k}\leqs K (jk)^{\b}\), with $\b\leqs \frac{1}{2}$. Then, for each \(c_0 \in X_{1}^{+} \) there is 
 one and only one density nonincreasing solution in
 \([0,T) \) such that \(c(0)=c_0 \).
\end{prop}

\begin{proof}
Suppose the initial value problem for (\ref{rbk}) with the initial condition $c(0)=c_0\in X_1^+$ has 
two density nonincreasing solutions, $c$ and $d$.\,\;
Let \( x(t):= c(t)-d(t)\) and \(M_{j,k}:= a_{j,k}(c_jc_k- d_jd_k)= a_{j,k}(c_jx_k+d_kx_j).\)
We shall prove that \(c\equiv d\) by establishing that, for some \(\alpha\), 
the sum $\sum_{j=1}^{\infty} j^{\alpha}|x_j|$
is identically zero. This will be achieved by deriving an inequality for this quantity and applying Gronwall's inequality.

So, let us consider   Proposition~\ref{prop2-2} with  \( m=1 \). From the above definitions we get
\[
 \sum_{j=1}^{n}g_jx_j=\int_{0}^{t}\biggl( \sum_{S_1}(g_{j-k}-g_j)M_{j,k}- \sum_{S_2}g_jM_{j,k}+\sum_{S_3}g_{j-k}M_{j,k}\biggr)ds.
\]
For each \( t\in [0,T) \) consider \( g_{j}=j^{\a}\sgn (x_{j}) \) for some $\a\geqs \b$ such that $\a+\b\leqs 1.$
We now estimate the sums over each of the $S_j$. For \( (j,k)\in S_{1} \) we get
\begin{eqnarray*}
(g_{j-k}-g_j)x_k &=&\bigl(
(j-k)^{\a}\sgn(x_{j-k})-j^{\a}\sgn(x_{j})
\bigr)x_{k}\\
&=&\bigl(
(j-k)^{\a}\sgn(x_{j-k})-j^{\a}\sgn(x_{j})
\bigr)\sgn(x_{k})\abs{x_{k}}\\
&=&\bigl(
(j-k)^{\a}\sgn(x_{j-k}x_{k})-j^{\a}\sgn(x_{j}x_{k})
\bigr)\abs{x_{k}}\\
&\leqs& ((j-k)^{\a}+j^{\a})\abs{x_{k}}\\
&\leqs& 2j^{\a}\abs{x_{k}},
\end{eqnarray*}
and, by a similar computation, 
\((g_{j-k}-g_j)x_j 
=
\bigl((j-k)^{\a}\sgn(x_{j-k})-j^{\a}\sgn(x_{j})
\bigr)x_{j}
= 
\bigl((j-k)^{\a}\sgn(x_{j-k}x_{j})-j^{\a}
\bigr)\abs{x_{j}}\;\,\leqs\;\, 0.
\)
Using these bounds and the assumptions on $\a$ and $\b$ we can estimate the sum over \(S_1\) as
\begin{eqnarray*}
\int_{0}^{t}\sum_{S_1}(g_{j-k}-g_j)M_{j,k} ds
&\leqs & 2K \int_{0}^{t}\sum_{j=2}^{n}\sum_{k=1}^{j-1}j^{\a}(jk)^{\b}c_{j}\abs{x_{k}}ds\\
&\leqs & 2K\norm{c_{0}}_1\int_{0}^{t}\sum_{k=1}^{n}k^{\b}\abs{x_{k}}ds.
\end{eqnarray*}
For \( (j,k)\in S_{2} \) we have \(
-j^{\a}\sgn (x_{j})x_{k}\leqs j^{\a}\abs{x_{k}},\) and 
\(j^{\a}\sgn (x_{j})x_{j} = j^{\a}\abs{x_{j}},\)
from which it follows that
\begin{eqnarray*} 
-\int_{0}^{t}\sum_{S_2}g_jM_{j,k} ds & \leqs 
& K\int_{0}^{t}\sum_{j=1}^{n} \sum_{k=j}^{\infty}j^{\a}(jk)^{\b}(c_{j}\abs{x_{k}}-d_{k}\abs{x_{j}})ds\\
&\leqs & K\norm{c_{0}}_1\int_{0}^{t}\sum_{k=1}^{\infty}k^{\b}\abs{x_{k}}ds.
\end{eqnarray*}
Finally, for \( (j,k)\in S_{3} \) we have the estimates 
\(
(j-k)^{\a}\sgn (x_{j-k})x_{k}\leqs n^{a}\abs{x_{k}},\) and
\((j-k)^{\a}\sgn (x_{j})x_{j}= n^{\alpha}\abs{x_{j}},
\)
from which it follows 
\begin{eqnarray}
\int_{0}^{t}\sum_{S_{3}}g_{j-k}M_{j,k}ds&\leqs &Kn^{\a}\int_{0}^{t}\sum_{j=n+1}^{\infty} 
\sum_{k=j-n}^{j-1}(jk)^{\b}(c_{j}\abs{x_{k}}+d_{k}\abs{x_{j}})ds\nonumber\\
&\leqs &Kn^{\a}\int_{0}^{t}\sum_{j=n+1}^{\infty}j^{\b}c_{j} \sum_{k=j-n}^{j-1}k^{\b}\abs{x_{k}}ds \; + \label{int1} \\
&\phantom{+}&+\;Kn^{\a}\int_{0}^{t}\sum_{j=n+1}^{\infty} j^{\b}\abs{x_{j}}\sum_{k=j-n}^{j-1}k^{\b}d_{k}ds.\label{int2}
\end{eqnarray}
For the double sum in (\ref{int1}),  using $\a+\b\leqs 1$ we conclude that
\[
\sum_{j=n+1}^{\infty}j^{\b}c_{j} \sum_{k=j-n}^{j-1}k^{\b}\abs{x_{k}}
\leqs \frac{1}{n^{\a}}\sum_{j=n+1}^{\infty}jc_{j} \sum_{k=j-n}^{j-1}k^{\b}\abs{x_{k}}
\leqs \frac{\norm{c_{0}}_1}{n^{\a}} \sum_{k=1}^{\infty} k^{\b}\abs{x_{k}}.
\]
For the double sum in (\ref{int2}) we have, again using $\a+\b\leqs 1$,
\[
\sum_{j=n+1}^{\infty} j^{\b}\abs{x_{j}}\sum_{k=j-n}^{j-1}k^{\b}d_{k} \leqs 
\frac{\norm{c_{0}}_1}{n^{\a}}\biggl(\sum_{j=n+1}^{\infty} j^{\b}c_{j}+\sum_{j=n+1}^{\infty} j^{\b}d_{j}\biggr),
\]
and therefore
\begin{multline*}
\int_{0}^{t}\sum_{S_{3}}g_{j-k}M_{j,k}ds\leqs K\norm{c_{0}}_1 \int_{0}^{t}\!\sum_{k=1}^{\infty} k^{\b}\abs{x_{k}}ds\; + \\ +
 K\norm{c_{0}}_1 \int_{0}^{t}\!\biggl(\sum_{j=n+1}^{\infty} j^{\b}c_{j}+\sum_{j=n+1}^{\infty} j^{\b}d_{j}\biggr)ds.
\end{multline*}
Combining the estimates on the three regions we get
\begin{align}
\sum_{j=1}^{n}j^{\a}\abs{x_{j}}
 &\leqs 4K\norm{c_{0}}_1\int_{0}^{t}\sum_{j=1}^{n}j^{\b}\abs{x_{j}} +
 K\norm{c_{0}}_1\int_{0}^{t}\biggl(\sum_{j=n+1}^{\infty}j^{\b}c_j+\sum_{j=n+1}^{\infty}j^{\b}d_j\biggr).\label{forgron}
\end{align}
By the definition of solution and the assumption $\b\leqs 1/2 <1$,
\[
\sum_{j=n+1}^{\infty}j^{\b}c_j \leqs\sum_{j=n+1}^{\infty}jc_j\to 0, \text{ as } n\to \infty, \text{ pointwisely,}
\]
and since \(c(\cdot) \) is a density nonincreasing solution,
\[
 \sum_{j=n+1}^{\infty}j^{\b}c_j(t) \leqs \sum_{j=n+1}^{\infty}jc_j(t) \leqs 
\sum_{j=1}^{\infty}jc_j(0)  = \norm{c_{0}}_1.
\]
Thus, by the dominated convergence theorem
\[
 \int_{0}^{t} \sum_{j=n+1}^{\infty}j^{\b}c_jds\to 0, \text{ as } n\to \infty,
\]
the same being valid for \(\sum_{j=n+1}^{\infty}j^{\beta}d_j\).

Therefore, letting \( n\to \infty\) in (\ref{forgron})  and using the assumption $\a\geqs \b$ in the right-hand side we obtain
\[
 \sum_{j=1}^{\infty}j^{\a}\abs{x_{j}}\leqs 4K\norm{c_{0}}_1\int_{0}^{t}\sum_{j=1}^{\infty}j^{\a}\abs{x_j}ds.
\]
Since \(x_j(0)=0 \), by Gronwall inequality we conclude
\[
  \sum_{j=1}^{\infty} j^{\a}\abs{x_j}=0, \qquad\forall t\in [0,T),
\]
and so 
\(
 x_j=0, 
\)
for all $t\in [0,T)$ and $j\in\Nb,$ thus proving uniqueness.\end{proof}

\begin{remark}\label{rembeta}
 Observe that, if $\b>1/2$, it is not possible to find numbers $\alpha$ such that $\a+\b \leqs 1$ and $\a\geqs\b$. 
It is this elementary technical reason that forced us to consider $\beta\leqs 1/2$ in the uniqueness
result in Proposition~\ref{uniqueness}, since in this case such numbers $\a$ obviously exist. 
It is  not presently clear if a uniqueness result is true in more 
general situations.
\end{remark}


\section{Differentiability and continuous dependence}
%
%
In Theorem \ref{theoexist} and Corollary \ref{admis} we proved that with the hypothesis that $a_{j,k}\leqs Kjk$, 
for each initial condition $c_{0}\in X^+_{1}$, there exists at least one admissible solution. With the stronger 
assumption that $a_{j,k}\leqs  K(jk)^{1/2}$, Theorem \ref{uniqueness} implies that there is a unique 
solution defined in $\left[0,\infty\right)$ and so, it has to be an admissible solution. We now address the 
issue of differentiability of such solutions.
\begin{theorem} If there is $K>0$ such that, for all $j,k\in\Nb$, $a_{j,k}\leqs Kjk$, and $c(\cdot)=(c_{j}(\cdot))$ 
is an admissible solution then, the functions $t\mapsto c_{j}(t)$, for $j\in\Nb$, and $t\mapsto \sum_{j=1}^{\infty}c_{j}(t)$ are 
continuously differentiable. Moreover, (\ref{rbk}) is satisfied for all $t\in\left[0,\infty\right)$ and
\begin{equation}\label{diffsum}
\frac{d}{dt}\sum_{j=1}^{\infty}c_{j}(t)=-\sum_{k=1}^{\infty}\sum_{j=1}^{k}a_{j,k}c_{j}(t)c_{k}(t),\quad\forall t\left[0,+\infty\right)\,.
\end{equation}
\end{theorem}
\begin{proof}
If a solution is admissible, we can use (\ref{normp}), like in Theorem \ref{theoexist} and Corollary \ref{admis}, 
to conclude that $\sum_{j=1}^{\infty}jc_{j}(t)$ is uniformly convergent in compact subsets of $\left[0,\infty\right)$. 
This, together with the assumption on the coefficients $a_{j,k}$ and the continuity of $c_{j}(\cdot)$, for $j\in\Nb$, 
allows us to prove the continuity of the right-hand side of (\ref{rbk}) in $t$, thus proving the continuous differentiability 
of each $c_{j}(\cdot)$, and the fact that (\ref{rbk}) is satisfied by $c(t)$, for all $t\in\left[0,+\infty\right)$. 

From the proof of Lemma \ref{momlem}, we already know that
\[
 \sum_{j=1}^{\infty}c_j(t) - \sum_{j=1}^{\infty} c_j(\tau)=-\int_{\tau}^{t}\sum_{k=1}^{\infty}\sum_{j=1}^{k}a_{j,k}c_{j}(s)c_{k}(s)\,ds\,.
\]
Since the double series of continuous functions in the r.h.s is uniformly convergent in each compact interval, the result follows.
\end{proof}

With an extra condition on the kinetic coefficients we can say more about the differentiability of the  moments for any solution:
\begin{prop}
If there is $K>0$ such that, for all $j,k\in\Nb$, $a_{j,k}\leqs K(jk)^{\beta}$, with $\beta\leqs \frac{1}{2}$, and 
$c(\cdot)=(c_{j}(\cdot))$ is a solution of (\ref{rbk}), and $(g_{j})$ is a nonegative sequence such that
$0\leqs g_{j}\leqs j$, and, for some positive constant $M$, $|g_{j}-g_{j-k}|\leqs Mk$ for $1\leqs k\leqs j$, then  for $m\in\Nb$, $t\mapsto\sum_{j=m}^\infty g_{j}c_{j}(t)$ is continuously differentiable and moreover, for $t\geqs 0$,
\begin{equation}\label{difffmoments}
\frac{d}{dt}\sum_{j=m}^\infty g_{j}c_{j}(t)=-\sum_{S_1(m,\infty)}(g_j -g_{j-k})W_{j,k} (c(t))-\sum_{S_2(m,\infty)}g_jW_{j,k}(c(t)) \,.
\end{equation} 
In particular the differential versions of (\ref{jwjj}), (\ref{varc}) and (\ref{oddweak}) hold true.
\end{prop}
\begin{Proof}
From the estimates \eqref{S2gj} and \eqref{S1gj} taken together with the property of uniform convergence of \(\sum_{j=1}^{\infty}jc_{j}\), 
it is clear that the convergence of the double series in (\ref{weakmoment}) is uniform so that the series in the integrals of the r.h.s. 
of that expression define continuous functions in $t$. This implies our result. 
\end{Proof}

With respect to the continuous dependence relatively to the initial conditions we prove the following partial result:
%
%
\begin{prop}
If \(a_{j,k}\leq K(jk)^\beta\) wth \(\beta\leqs \frac{1}{2}\), if \(\alpha+\beta\leqs 1\) and if \(c\) and \(d\) are solutions of \eqref{rbk} 
satisfying \(c(0)=c_{0}\) and 
\(d(0)=d_{0}\) then, for each \(t\geqs 0\), there is a positive \(C(t,\|c_{0}\|_{1})\) such that
\begin{equation}\label{contalpha}
\|c(t)-d(t)\|_{\alpha}\leqs C(t,\|c_{0}\|_{1}) \|c_{0}-d_{0}\|_{\alpha}.
\end{equation}
\end{prop}
\begin{Proof}
By writing
 \begin{align*}
 c_j(t)&=c_{0j}+\int_{0}^{t}\left[\sum_{k=1}^{\infty}a_{j+k,k}c_{j+k}(s)c_{k}(s)- 
 \sum_{k=1}^{\infty}a_{j,k}c_{j}(s)c_{k}(s)\right]ds,\\
 d_j(t)&=d_{0j}+\int_{0}^{t}\left[\sum_{k=1}^{\infty}a_{j+k,k}d_{j+k}(s)d_{k}(s)- 
 \sum_{k=1}^{\infty}a_{j,k}d_{j}(s)d_{k}(s)\right]ds.
\end{align*}
and defining \(x(t)=c(t)-d(t)\) we perform the same estimates as in the proof of Proposition~\ref{uniqueness} to obtain, this time, 
\begin{multline*}
	\sum_{j=1}^\infty j^\alpha |x_{j}(t)| \leqs \sum_{j=1}^\infty j^\alpha |x_{j}(0)|+4K\|c_{0}\|_{1}\int_{0}^{t} \sum_{j=1}^{\infty}j^\beta |x_{j}|+\\+
	K\|c_{0}\|_{1}\int_{0}^{t}\sum_{j=n+1}^{\infty}j^\beta c_{j}+K\|d_{0}\|_{1}\int_{0}^{t}\sum_{j=n+1}^{\infty}j^\beta d_{j}\,,
\end{multline*}
instead of \eqref{forgron}. Hence, by making \(n\to\infty\) and using the same arguments as in that proof, we obtain,
\[
\sum_{j=1}^\infty j^\alpha |x_{j}(t)| \leqs \sum_{j=1}^\infty j^\alpha |x_{j}(0)| + 4K\|c_{0}\|_{1}\int_{0}^{t} \sum_{j=1}^{\infty}j^\alpha |x_{j}|\,.
\]
By using the Gronwall lemma estimate \eqref{contalpha} follows.
\end{Proof}
\begin{remark}
Here we recall the Remark~\ref{rembeta} with respect to our restricting hypothesis on the kinetic coefficents and on the growing rate of \((g_{j})\)\,.
Notice that we were not able to prove continuous dependence on the initial conditions with respect to the \(X_{1}\) norm, except in the case where 
\(a_{j,k}\leqs K\), for all \(j,k\in\Nb \). At the moment it is not clear to us whether this is an essential feature of the RBK equation or if 
it is just a technical limitation due to the methods we have used. It is possible that, in a more general case, 
instead of a continuity property based on a norm estimate like 
\eqref{contalpha}, we can prove an upper semicontinuity property for admissible solutions similar to that in \cite[Theorem 5.4]{BC}.
\end{remark}

\section{Some invariance properties of solutions}\label{invariance}

From the physical process that we are modelling (see scheme in Fig.~\ref{cluster_eating}) it is natural to expect that if initially
 there are no clusters of size larger than $p$, then none will be produced afterward. This is established next.

\begin{prop}\label{prop5-1}
Assume the Cauchy problems for (\ref{rbk}) have unique solutions.
 Then, for every $p\in \Nb,$ the sets 
\[
 X^{\leqs p} := \{c\in X^+_1|\;  c_j=0,\,\forall j>p\}
\]
are positively invariant for (\ref{rbk}). 
\end{prop}

\begin{proof}
Let $c$ be a solution of (\ref{rbk}) such that $c(\tau) = c_0\in X^{\leqs p}$, for some $\tau\geqs 0.$

Let $c^p(\cdot)$ be the unique solution of the $p$-dimensional Cauchy problem
\[
\begin{array}{l}
 \dot{c}_j^p = \displaystyle{\sum_{k=1}^{p-j}W_{j+k,k}(c^p) - \sum_{k=1}^{p}W_{j,k}(c^p)} \\
 c_j^p(\tau)\rule{0mm}{4mm} = c_{0\,j},
\end{array}
\]
for $j = 1,\ldots, p$ (with the first sum defined to be zero if $j=p.$) 
Then, the function $(c_1^p, c_2^p, \ldots, c_p^p, 0, 0, \ldots)$ is a solution of the infinite dimensional system (\ref{rbk}) and,
by uniqueness, it must be the solution $c$. Therefore, for all $t\geqs\tau$, we have $c_j(t)=0$ when $j=p+1, p+2, \ldots$,
that is,  $c(t)\in X^{\leqs p}$ for all $t\geqs\tau,$ which proves the result.
\end{proof}

This invariance property also occur in {\it fragmentation\/} equations: 
if the initial
distribution of clusters has no clusters with size larger than $p$, then they cannot be produced by fragmentation of those 
(smaller) ones that are
initially present; a reasonable enough result. Invariant sets for coagulation equations
with Smoluchowski coagulation processes 
are of a different kind but one can also characterize them without much difficulty \cite{dC}.
In fact, we can use a similar approach also in this case to characterize the positivity
properties of the cluster distribution (i.e., the subscripts $j$ for which $c_j(t)>0$)
in terms of those same properties for the initial data. Let us first introduce some notation.

For a solution $c=(c_j)$ to (\ref{rbk}), denote by $P := \{j\in\Nb\,|\, c_j(0)>0\}$ the set 
of integers (finite or infinite) describing the positive components of the initial condition $\bigl(c_j(0)\bigr),$
and let $\gcd(P)$ be the greatest common divisor of the elements of $P$.
Define $\mathcal{J}(t):= \{j\in\Nb\,|\,c_j(t)>0 \},$ the set of indices for which the component of the solution is
positive at the instant $t.$ Naturally, $P=\mathcal{J}(0).$ Now we have the following result: 

\begin{prop}\label{prop5-2}
Assume uniqueness of solution to initial value problems for (\ref{rbk}) holds.
Let $\#P>1$. Then, 
\[
 \gcd(P)=m \Longrightarrow \mathcal{J}(t) = m\Nb\cap [1, \sup P],\quad \forall t>0.
\]
\end{prop}

\begin{proof}
We first remark that, by uniqueness of solution and the form of the system (\ref{rbk}), if, for any $j$, one has 
$c_j(s)=0$ for all $s$ in a nondegenerate interval $[\t, t]$, then $c_j$ is identically zero for all times.

In order to prove the proposition it is convenient to write system (\ref{rbk}) in a different form, similar to what was done 
for the Smoluchowski's coagulation equation in \cite{dC}. Define
\[
 R_j(t):= \sum_{k=1}^{\infty}a_{j+k,k}c_{j+k}(t)c_k(t), \;\, \varphi_j(t):=\sum_{k=1}^{\infty}a_{j,k}c_k(t), 
\;\, E_j(t):=\exp\biggl(\int_0^t\!\!\varphi_j(s)ds\biggr).
\]
Observe that all these functions are nonnegative, for nonnegative solutions and for all $t$ (and furthermore $E_j(t)\geqs 1$). 
Using them, write (\ref{rbk}) as
\[
 \dot{c}_j = R_j - c_j\varphi_j
\]
and apply the variation of constants formula to get, for all $t\geqs \t\geqs 0,$
\begin{equation}
 c_j(t)E_j(t) = c_j(\t)E_j(\t) + \int_{\t}^t\!E_j(s)R_j(s)ds. \label{fvc}
\end{equation}
Equation (\ref{fvc}) allows the following conclusions to be immediately drawn:
\begin{description}
\item[{\rm (i)}] if $c_j(\t)>0$, then $c_j(t)>0,$ for all $t>\t$, or, equivalently, $\mathcal{J}(t)\supseteq \mathcal{J}(\t)$
for all $t>\t,$ and in particular $\mathcal{J}(t)\supseteq P,$ for all $t>0.$
\item[{\rm (ii)}] due to the definition of $R_j(\cdot)$, if $\ell_2>\ell_1$ are two numbers in $\mathcal{J}(\t)$, then $\ell_2-\ell_1\in
\mathcal{J}(t)$, for all $t>\t.$
\item[{\rm (iii)}] by (i) and (ii) one concludes that, if $p_1, \ldots, p_n \in \mathcal{J}(\t),$ 
(assuming, without loss of generality, 
that $p_i<p_j$ for $i<j$), then, for all integers $m_i\in\Zb$, 
we have $m_1p_1 + \ldots +m_np_n \in \mathcal{J}(t)$, for $t>\t$, provided the integers $m_i$ are such that 
$1\leqs m_1p_1 + \ldots +m_np_n \leqs p_n.$
\end{description}

Suppose that $1 < \#P< \infty$. Let us write $P=\{p_1, \ldots, p_n\}$. By B\'ezout's lemma in
elementary number theory \cite[Chapter 1]{S} we conclude that if $\gcd(P) =m$ then the smallest positive value 
of  $m_1p_1 + \ldots +m_np_n$ is $m$ and all other
larger values are multiples of $m.$ Hence, this result, (iii), and Proposition~\ref{prop5-1} imply 
that $\mathcal{J}(t) \supseteq m\Nb\cap [1, \sup P],$ for all $t>0,$ when the initial data is finitely supported.
It is clear that, if $m=1,$ then equality holds and the proof is complete for finitely supported initial data.

In order to complete the proof we now need to prove that, when $m>1,$ we also have $\mathcal{J}(t) \subseteq m\Nb\cap [1, \sup P].$ 

We first note that $\gcd(P)=m\Rightarrow P\subseteq m\Nb\cap [1, \sup P].$ This is obvious since the assumption implies that every
element of $P$ is a multiple of $m$ (hence an element of $m\Nb$) and, naturally, it is not bigger than the supremum of $P$.

Let us now prove the result. Note that it suffices to prove that, for any $q\in\Nb,$
\[
 q\not\in m\Nb\cap [1, \sup P] \Rightarrow c_q(t)=0, \forall t>0.
\]
Let $d=(d_j),$ $j\in\Nb,$ be defined by $d_j=0$ if $j>\sup P,$ and, for $j\leqs \sup P,$ let $d_j$ be given by the solution of the 
ordinary differential equation
\[
 \begin{cases}
  \dot{d}_j = 0  &\text{ if $j\not\in m\Nb\cap [1, \sup P]$,}\\
  \dot{d}_j = \displaystyle{\sum_{k=1}^{p-j}a_{j+k,k}d_{j+k}d_k- d_j\sum_{k=1}^{\infty}a_{j,k}d_k} &\text{ otherwise,}
 \end{cases}
\]
where $p:=\sup P$, with initial condition 
\[
 d_j(0)=c_j(0),\quad, j=1, \ldots, p.
\]

Let $j\not\in m\Nb\cap[1, \sup P].$ We know that $j\not\in P$ and therefore $d_j(0)=c_j(0)=0$. Thus, by the differential 
equation, $d_j(t)=0,$ for all $t\geqs 0.$ On the other hand, if $j\not\in m\Nb\cap [1, \sup P]$ it is not possible that
both $k$ and $j+k$ belong to $m\Nb\cap [1, \sup P],$ for every $k.$ Thus, we conclude that
\[
 \sum_{k=1}^{p-j}a_{j+k,k}d_{j+k}d_k =0.
\]
Moreover, since $d_j=0$ for all $t$, we also have
\[
 d_j\sum_{k=1}^pa_{j,k}d_k = 0,
\]
and hence $d$ is also solution of the system. 
\[
 \dot{d}_j = \sum_{k=1}^{p-j}a_{j+k,k}d_{j+k}d_k -d_j\sum_{k=1}^pa_{j,k}d_k,\quad j=1, \ldots, p,
\]
with initial condition $d_j(0)=c_j(0).$
Therefore, by uniqueness, $c=d$ and we conclude that $q\not\in m\Nb\cap [1, \sup P] \Rightarrow c_q(t)=d_q(t)=0,\; \forall t\geqs 0.$

Suppose now that $\#P=\infty$. Let $\gcd(P)=m$. By \cite[Proposition 5]{dC2} 
there exists a finite subset $P_n\subset P,$ with $\#P_n >1,$ 
such that $\gcd(P_n)=m.$ Apply now the above argument to $P_n$ instead of $P$. This concludes the proof.  
\end{proof}

The previous proof does not hold if $\#P=1.$ In that case, a peculiar behaviour occurs, not exhibited either by the
usual Smoluchowski's equations, or by the fragmentation equations.

\begin{prop}\label{prop5-3}
Assume uniqueness of solutions to initial value problems for (\ref{rbk}) holds.
If a solution to (\ref{rbk}) starts monodisperse, it stays monodisperse for all later times, or, 
in the notation used above, \[\#P=1 \Rightarrow \mathcal{J}(t) = P,\quad \forall t>0.\]
\end{prop}

\begin{proof}
 Let $c_j(0)=\lambda \delta_{j,p}, \lambda >0$, for some positive integer $p$. Then, $c(0)\in X^{\leqs p}$ and, by 
 Proposition~\ref{prop5-1}, 
$c(t)\in X^{\leqs p}$, for all $t\geqs 0.$ Therefore $c_j(t)=0$ for all $j\geqs p+1$ and, for $j=1, \ldots, p$,  $c_j(t)$ 
is given by the 
$p$-dimensional system considered above. Now let $c_j(t)=\alpha(t)\delta_{j,p}$. Obviously we have, for all $j=1,\ldots, p,$
\[
\displaystyle{\sum_{k=1}^{p-j}W_{j+k,k}(c)} = 0, 
\qquad\qquad \displaystyle{\sum_{k=1}^{p}W_{j,k}(c)} = a_{j,p}\alpha(t)^2\delta_{j,p}.
\]
Thus, $c(\cdot)$ will be a solution of (\ref{rbk}) with initial condition $\bigl(\lambda\delta_{j,p}\bigr)$ 
if and only if $\alpha(\cdot)$ solves
\[
 \begin{cases}
  \dot{\alpha} = -a_{p,p}\alpha^2 &\\
 \alpha(0)= \lambda.&
 \end{cases}
\]

Hence, solving this initial value problem and substituting back into the expression for $c(\cdot)$ we obtain the following solution
of (\ref{rbk}) 
\[
 \begin{cases}
  c_p(t) = \displaystyle{\frac{\lambda}{1+\lambda a_{p,p}t}} & \\
 c_j(0) = 0\;\,\text{for $j\neq p$}.&
 \end{cases}
\]

By uniqueness, it is the only solution satisfying the initial condition $c_j(0)=\lambda \delta_{j,p},$ which proves the result.
\end{proof}


\section{On the long-time behaviour of solutions}

In this section we start the investigation of the long-time behaviour of solutions.

Having present the physical process under consideration (see scheme in Fig.~\ref{cluster_eating}), it is natural to
expect that, under rather mild conditions, all solutions will converge pointwise to zero as $t\to\infty.$ For 
solutions in $X_1$ this can be phrased by saying that solutions converge to zero in the weak-$\ast$ sense \cite[page 672]{BCP}.

The result is given in the following proposition, the proof of which is rather easy and follows the same ideas used to prove
the same result in the {\it fragmentation\/} equation \cite[Theorem 4.1]{CdC}, which we reproduce here for the sake of completeness
of presentation.

\begin{prop}\label{prop6-1}
Let $a_{j,k}\leqs Kjk,$ and let $c$ be a solution of (\ref{rbk}) with initial condition $c(t_0)=c_0\in X^+_1$, defined on
$[t_0,\infty).$ Assume that $a_{j,j}>0,$ for all $j$. 
Then, for all $j\in \Nb$ it holds that $c_j(t)\to 0$ as $t\to\infty.$
\end{prop}

\begin{proof}
 Let $t> \tau,$ and consider the moments' equation (\ref{limSSS}) with $g_j\equiv 1,$ 
  \[
   \sum_{j=m}^{\infty}c_j(t) - \sum_{j=m}^{\infty}c_j(\tau) 
   = - \lim_{n\to\infty}\int_{\tau}^t\sum_{j=m}^{n}\sum_{k=j-m+1}^{\infty}a_{j,k}c_j(s)c_k(s)ds
\leqs 0,
  \]
from which we conclude that $t\mapsto \sum_{j=m}^{\infty}c_j(t)$ is a monotonic nonincreasing function. As it is bounded below (by zero) 
it must converge to some constant $p_m^{\ast}\geqs 0.$ Since $\sum_{j=m}^{\infty}c_j(t)\geqs \sum_{j=m+1}^{\infty}c_j(t)$ we have 
$p_m^{\ast}\geqs p_{m+1}^{\ast}.$ Then, for all $m\in\Nb$ we have
\[
 c_m(t) = \sum_{j=m}^{\infty}c_j(t)-\sum_{j=m+1}^{\infty}c_j(t) \xrightarrow[t\rightarrow \infty]{} p_m^{\ast}-p_{m+1}^{\ast} =: 
c_m^{\ast}\geqs 0.
\]
Now consider the case $m=1$ and $t=\tau + 1.$ Applying limits $\tau\to\infty$ to the moments' equation, we have
\[\renewcommand{\arraystretch}{2.0}
 \begin{array}{rcl}
  \displaystyle{\sum_{j=1}^{\infty}c_j(\tau+1) - \sum_{j=1}^{\infty}c_j(\tau)} &=& \displaystyle{
  - \lim_{n\to\infty}\int_{\tau}^t\sum_{j=1}^{n}\sum_{k=j}^{\infty}a_{j,k}c_j(s)c_k(s)ds}\\
& \Big\downarrow\rlap{$\scriptstyle\tau\to\infty$}  & \\
0 & = & \displaystyle{- \lim_{\tau\to\infty}\lim_{n\to\infty}\int_{\tau}^t\sum_{j=1}^{n}\sum_{k=j}^{\infty}a_{j,k}c_j(s)c_k(s)ds}.
 \end{array}
\]
Suppose there exists an integer $p\in\Nb$ such that $c_p^{\ast}>0.$ Let $\alpha_p\in (0, c_p^{\ast}).$ Then
\begin{eqnarray*}
 0 & = & \displaystyle{\lim_{\tau\to\infty}\lim_{n\to\infty}\int_{\tau}^{\tau +1}\sum_{j=1}^{n}\sum_{k=j}^{\infty}a_{j,k}c_j(s)c_k(s)ds}\\
& \geqs &  
 \displaystyle{\lim_{\tau\to\infty}\lim_{n\to\infty}\int_{\tau}^{\tau+1}\sum_{j=1}^{n}a_{j,j}\bigl(c_j(s)\bigr)^2ds}\\
& \geqs & \displaystyle{\lim_{\tau\to\infty}\int_{\tau}^{\tau+1}a_{p,p}\bigl(c_p(s)\bigr)^2ds}\;\;>\;\;
 \displaystyle{\lim_{\tau\to\infty}a_{p,p}\alpha_p^2} \; = \; a_{p,p}\alpha_p^2\;\;>\;\; 0,
\end{eqnarray*}
and this contradiction proves that $c_p^{\ast}=0,$ for all values of $p.$
\end{proof}


\section{On the scaling behaviour of solutions}
In this section we begin the study of the scaling behaviour of the solutions of (\ref{rbk}) in the particular case when 
$a_{j,k}=1, \text{for }j,k\in\Nb,$
in which case the system turns into
\begin{equation}\label{rbk1}
\dot{c}_{j}=\sum_{k=1}^{\infty}c_{j+k}c_{k}-c_{j}\sum_{k=1}^{\infty}c_{k}\,,\quad j=1,2,\dots
\end{equation}

This study is strongly motivated by similar studies on the scaling behaviour of coagulation-fragmentation equations and other 
related equations \cite{LM2004}.  Most of the  results in those works are consequences of the application of tools 
based on the Laplace transform. Here we have an entirely new situation since the production term (the first one on the r.h.s. 
of (\ref{rbk1})) is not of convolution type and hence Laplace transform methods are not useful.  

In the first place, we can draw consequences from the differential version of  (\ref{varc}) and (\ref{oddweak}) about the 
typical time scales for the cluster eating equation. Some of these conclusions were already formally obtained in \cite{RBK} and 
here we reproduce part of their arguments. In that work the authors were led to interesting novel features about 
the evolution of the numbers of clusters of odd and even sizes, which have no parallel in 
the usual Smoluchowski's coagulation-fragmentation equation.
These were already pointed out in Section~\ref{invariance}. Let us define, for $t\geq 0$,
\[
\nu(t):=\sum_{j=1}^{\infty}c_{j}(t),\qquad \nu_{\text{odd}}(t):=\sum_{j=1}^{\infty}c_{2j-1}(t)\,.
\] 
With our choice of the kinetic coefficients, equation (\ref{varc}) becomes
\begin{equation}\label{nudot}
\dot{\nu}=-\frac{1}{2}\nu^2-\frac{1}{2}\sum_{j=1}^{\infty}c_{j}^2\,,
\end{equation}
while equation (\ref{oddweak}) turns into
\begin{equation}\label{nuodddot}
\dot{\nu}_{\text{odd}}=-\nu_{\text{odd}}^2\,.
\end{equation}
Equation (\ref{nuodddot}) reflects the fact that the number of even size clusters does not affect the evolution of the number of odd size clusters. 
This is intuitively clear since the only interaction that changes the net amount of odd sized clusters is the reaction
between two odd size clusters, which produces 
an even size one. Solving (\ref{nuodddot}) we get
\[
\nu_{\text{odd}}(t)=\frac{\nu_{\text{odd}}(0)}{1+\nu_{\text{odd}}(0)t},\quad\text{for } t\geqs 0\,.
\]
On the other hand, (\ref{nudot}) implies that 
\[
-\nu^2\leqs\dot{\nu}\leqs -\frac{1}{2}\nu^2,
\]
and hence
\[
\frac{\nu(0)}{1+\nu(0)t}\leqs \nu(t) \leqs \frac{\nu(0)}{1+\frac{\nu(0)}{2}t}\,.
\]
These results give us a typical time scale of $t^{-1}$ which was already seen in the monodisperse solutions in section \ref{invariance}. From 
the above computations,
\[
\lim_{t\to+\infty} t\nu_{\text{odd}}(t)=
\begin{cases}
0,&\text{ if }\nu_{\text{odd}}(0)=0,\\
1,&\text{ if }\nu_{\text{odd}}(0)\not=0\,,
\end{cases}
\]
\[
\liminf_{t\to+\infty}t\nu(t)\geqs 1\,,\qquad \limsup_{t\to+\infty}t\nu(t)\leqs 2\,.
\]

We now turn our attention to the study of self-similar solutions. We call  a solution $c(\cdot)$ self-similar if there are 
functions $\zeta, \eta, \phi$ such that,
\begin{equation}\label{selfsim}
c_{j}(t)=\frac{1}{\zeta (t)}\phi\left(\frac{j}{\eta (t)}\right)\,,\quad\text{ for }j\in\Nb,\; t\geqs 0\,.
\end{equation}
We first remark that by virtue of the results obtained in Section~\ref{invariance} with respect to the monodisperse 
solutions (which are a trivial case of self-similar solutions), no universal behaviour is to be expected with respect to all 
the nonnegative solutions of (\ref{rbk1}). In other words, for $j$ and $t$ large, the solutions will not have an unique 
asymptotic self-similar behaviour independently of their initial conditions.

In order to find self-similar solutions we adopt an ansatz suggested in \cite{K} in a related situation: find differentiable 
functions $A$ and $\alpha$ such that there is a solution in the form 
\[
c_{j}(t)=A(t)\alpha(t)^j\,.
\]
By plugging this into (\ref{rbk1}) we obtain
\[
\dot{A}\alpha^j+jA\alpha^{j-1}\dot{\alpha}=-\frac{\alpha A^2}{1-\alpha^2}\alpha^j\,.
\] 
This clearly implies that $\alpha$ is a constant in $[0,1)$, while $A$ must satisfy
\[
\dot{A}=-\frac{\alpha}{1-\alpha^2}A^2\,.
\]
By integrating this ODE with initial condition $c_{j}(0)=A_{0}\alpha^j$, for $j\in\Nb$, with $A_{0}>0$, $\alpha\in  [0,1)$, 
 we obtain the solution
\begin{equation}\label{selfsimsol}
c_{j}(t)=\frac{A_{0}\alpha^j}{1+\beta t},\quad j\in\Nb, t\geqs 0\,,
\end{equation}
where $\beta:= \displaystyle\frac{A_{0}\alpha}{1-\alpha^2}$. This fits our definition of self-similar solution given by (\ref{selfsim}) with,
\[
\zeta (t) = 1+\beta t,\qquad \eta (t) = 1,\qquad \phi(x)=A_{0}\alpha^{x}\,.
\]
This confirms again the typical time scale of $t^{-1}$ and, in fact, for each such solution and for $j\in\Nb$,
\begin{equation}\label{asymself}
\lim_{t\to+\infty} tc_{j}(t) = (1-\alpha^2)\alpha^{j-1},
\end{equation}
and
\begin{equation}\label{asydens}
\lim_{t\to+\infty} t\nu(t)=1+\alpha\,.
\end{equation}
Therefore, we can consider (\ref{selfsimsol}) as a family of self-similar solutions, 
whose scaling behaviour depends only upon the rate $\alpha$ of exponential decreasing 
of the initial condition. For each $\alpha\in\left[0,1\right)$ there is one such 
solution with number of cluster density $\nu  (t)$ behaving like in (\ref{asydens}). 
At present the answers to the following questions about the asymptotic behaviour 
of the solutions of the cluster eating equation (\ref{rbk1}) are still unknown:
\begin{enumerate}
\item Besides (\ref{selfsimsol}) are there other strictly positive self-similar solutions? 
\item Are there other solutions behaving asymptotically like the solutions (\ref{selfsimsol}) 
in the sense that they satisfy (\ref{asymself})?
\item If the question above has positive answer and if we know that the total number of clusters 
has the behaviour (\ref{asydens}) for a known $\alpha$, can we guarantee that the solution itself 
behaves like (\ref{asymself}) or are there other types of asymptotic behaviour?
\end{enumerate}

We hope to return to these questions, as well as to a further exploration of the RBK system in the near future.


\bibliographystyle{amsplain}

\begin{thebibliography}{10}

\bibitem{BC} J.M. Ball, J. Carr, \textit{The discrete coagulation-fragmentation equations: existence, 
uniqueness, and density conservation}, 
J. Stat. Phys. \textbf{61} (1990) 1/2, 203--234.

\bibitem{BCP} J.M. Ball, J. Carr, O. Penrose, \textit{The Becker-D\"oring cluster equations: basic properties and asymptotic 
behaviour of solutions}, 
Commun. Math. Phys. \textbf{104} (1986) 657--692.

\bibitem{CdC} J. Carr, F.P. da Costa \textit{Asymptotic behaviour of solutions to the coagulation-fragmentation equations.
II. Weak fragmentation}, J. Stat. Phys. \textbf{77} (1994) 1/2, 89--123.

\bibitem{dC} F.P. da Costa \textit{On the positivity of solutions to the Smoluchowski equations}, Mathematika \textbf{42} 
(1995) 1/2, 406--412.

\bibitem{dC2} F.P. da Costa \textit{On the dynamic scaling behaviour of solutions to the discrete Smoluchowski equations}, 
Proc. Edinburgh Math. Soc. \textbf{39} (1996) 547--559.

\bibitem{CPRS} F.P. da Costa, J.T. Pinto, H.J. van Roessel, R. Sasportes, \textit{Scaling behaviour in 
a coagulation-annihilation model and Lotka-Volterra competition systems}, 
J. Phys. A: Math. Teor., \textbf{45} (2012) 285201.

\bibitem{war} I. Ispolatov, P.L. Krapivsky, S. Redner, \textit{War: The dynamics of vicious civilizations}, 
Phys. Rev. E \textbf{54} (1996) 1274-1289

\bibitem{K} P.L. Krapivsky, \textit{Nonuniversality and breakdown of scaling in two-species aggregation with annihilation}, 
 Physica A \textbf{198} (1993), 135--149.

\bibitem{L2002} Ph. Lauren\c{c}ot, \textit{The discrete coagulation equations with multiple fragmentation},
Proc. Edinburgh Math. Soc. \textbf{45} (2002) 67--82.

\bibitem{LM2004} Ph. Lauren\c{c}ot, S. Mischler, \textit{On coalescence 
equations and related models}; in: P. Degond, L. Pareschi, G. Russo 
(Eds.), \textit{Modelling and computational methods for kinetic 
equations}, Birkh\"auser, Boston, 2004, pp. 321--356.

\bibitem{RBK} S. Redner, D. Ben-Avraham, B. Kahng, \textit{Kinetics of `cluster eating'}, J. Phys. A: Math. Gen. 
\textbf{20} (1987), 1231--1238.

\bibitem{S} W. Sierpi\'nski, \textit{Elementary number theory}, Polska Akademia Nauk Monografie 
Matematyczne, Tom 46, Pa\'nstwowe Wydawnictwo Naukowe, Warszawa, 1964.

\bibitem{ZY} L. Zhang, Z.R. Yang,, \textit{A solvable aggregation-annihilation chain model with $n$ species}, 
 Physica A \textbf{237} (1997), 441--448.

\end{thebibliography}

\end{document}